\documentclass[12pt,reqno]{amsart}

\input RCT.sty

\title{Krawtchouk transforms and Convolutions}

\author{Philip Feinsilver}
\address{Department of Mathematics\\
Southern Illinois University\\
Carbondale, IL. 62901, U.S.A.}
\author{Ren{\'e} Schott}
\address{IECL and LORIA\\
Nancy-Universit\'e, Universit\'e de Lorraine\\
BP 239, 54506 Vandoeuvre-l\`es-Nancy, France.}

\begin{document}
\maketitle

\begin{abstract}    
We put together the ingredients for an efficient operator calculus based on Krawtchouk polynomials, including Krawtchouk transforms and 
corresponding convolution structure which provide an inherently discrete alternative to Fourier analysis. 
In this paper, we present the theoretical aspects and some basic examples. 
\end{abstract}

\thispagestyle{empty}
\begin{section}{Introduction}

Krawtchouk polynomials are part of the legacy of Mikhail Kravchuk (Krawtchouk), see \cite{VirKrav} as a valuable resource about his life and work, including
developments up through 2004 based on his work. Krawtchouk polynomials appear in diverse areas of mathematics and science. Important
applications such as to image processing \cite{Y} are quite recent and indeed are current.\bigskip

We cite \cite{APW, Lo2, SA} where Krawtchouk polynomials are used as the foundation for discrete models of quantum physics. And they appear
naturally when studying random walks in quantum probability \cite{MR1873669,MR2102950}.  \bigskip

After this Introduction, we continue with the probabilistic construction of Krawtchouk polynomials. They appear as the elementary
symmetric functions in the jumps of a random walk, providing a system of martingales based on the random walk.
Some fundamental recurrence relations are presented as well. The construction immediately yields their orthogonality relations.
Alternative probabilistic approaches to ours of \S2 are to be found in 
\cite{MR0486694,MR2407600,MR724060,MR2470507}.   \bigskip

Section 3 provides the linearization
and convolution formulas that are the core of the paper. They are related to formulas found in \cite{MR0320392,MR523857}. 
The next section, \S4, specializes to the case of a symmetric random walk, where the formulas simplify considerably. \bigskip

In Section 5 we introduce shift operators and use them to develop a computationally effective approach to finding transforms and convolutions.
This differs from our principal work with operator calculus \cite{FS2} and recent approach to Krawtchouk transforms \cite{AISC2010} 
and is suitable for numerical as well as symbolic computations.  \bigskip

The article concludes with \S6 which presents special bases in which the Krawtchouk matrices are anti-diagonal. These basis functions
have limited support and look to be useful in implementing filtering methods in the Krawtchouk setting.
\end{section}

\begin{section}{Combinatorial and probabilistic basis. Main features.}
Consider a collection of $N$ bits $B=\{0,1\}$ or signs $S=\{-1,1\}$. Correspondingly, we let $j$ denote
the number of $0$'s or $-1$'s. And we denote the sum in either case by $x$. So $x=N-j$ for bits, $x=N-2j$ for signs.
Order the elements of $B$ or $S$ and denote them by $X_i$, $1\le i\le N$. We can encode this information in the generating function
$$G_N(v)=\prod_{i=1}^N (1+vX_i)$$
Now introduce a binomial probability space with the $X_i$ a sequence of independent, identically distributed Bernoulli variables.
With $p$ the probability of ``success", $q=1-p$, the centered random variables are distributed as follows: \bigskip

Bits: $\displaystyle X_i-p=\begin{cases}\ \  q,&\text{with probability }p\\ -p,&\text{with probability }q \end{cases} $\bigskip

Signs: $\displaystyle X_i-(p-q)=\begin{cases}\ \   2q,&\text{with probability }p\\ -2p,&\text{with probability }q \end{cases} $\bigskip

To get a sequence of orthogonal functionals of the process we redefine
\begin{equation}\label{expm}
G_N(v)=\prod_{i=1}^N (1+v(X_i-\mu))=\sum_n v^n\,k_n(j,N)
\end{equation}
where $\mu$ is the expected value of $X_i$. We see that the two cases differ effectively as a rescaling of $v$. 
To see how this comes about, consider general Bernoulli variables $X_i$ taking values $a$ and $b$ with probabilities $p$ and $q$ respectively.
Then the centered variables take values 
$$\begin{cases}
a-(pa+qb)=\lambda q, &\text{with probability }p\\ b-(pa+qb)=-\lambda p, &\text{with probability }q
\end{cases}
$$
where $\lambda=a-b$. We can take as standard model $b=0$ and $a=\lambda$. Then 
$$\mu=\lambda p\quad\text{and}\quad \sigma^2=\lambda^2pq$$
are the mean and variance of $X_i$. Thus, $G$ has the form
$$G_N(v)=(1+\lambda qv)^{N-j}(1-\lambda pv)^j=\sum_n v^n\,k_n(j,N)$$
with $j$ counting the number of $0$'s and
$$k_n(j,N)=\lambda ^n\,\sum_i \binom{N-j}{n-i}\binom{j}{i}(-1)^i p^i q^{n-i}$$
These are polynomials in the variable $j$, \textit{Krawtchouk polynomials}. We define a corresponding matrix 
$$\Phi_{ij}^{(N)}=k_i(j,N)$$
which acts as a transformation on $\RR^{N+1}$, which we consider as the space of functions defined on the set $\{0,1,\ldots,N\}$.
The generic form, equation \eqref{expm}, is convenient for revealing and proving properties of the Krawtchouk polynomials,
and of the transform $\Phi$. \bigskip

We review here some principal features of this construction \cite{MR1873669,MR2102950}. \bigskip

\begin{remark} Denote expectation with respect to the underlying binomial distribution with angle brackets:
$$\avg f(X).=\sum_j \binom{N}{j}f(j)\,p^{N-j}q^j$$
and corresponding inner product $\avg f,g.=\avg f(X)g(X).$ .
\end{remark}

\begin{subsection}{Martingale property}
Since the $X_i$ are independent and $X_i-\mu$ has mean zero, we have the martingale property
$$E(G_{N+1}|\F_N)=\avg (1+v(X_{N+1}-\mu).\, G_N=G_N$$
where $\F_N$ is the $\sigma$-field generated by $\{X_1,\ldots,X_N\}$. Thus each coefficient $k_n(j,N)$ is a martingale, where $j$
denotes the number of 0's in the random sequence of 0's and 1's which is the sample path of the underlying Bernoulli process. This gives immediately
\begin{proposition} \text{Martingale recurrence}
$$k_n(j,N)=p\,k_n(j,N+1)+q\,k_n(j+1,N+1)$$
\end{proposition}
One can derive this purely algebraically by the Pascal recurrences presented in the next paragraph.
\end{subsection}
\begin{subsection}{Pascal recurrences and square identity}
As is evident from the form of the generating function $G$, we have recurrences analogous to the Pascal triangle for binomial coefficients.
\begin{proposition}{\sl Pascal recurrences} \bigskip

1. $k_n(j,N+1)=k_n(j,N)+\lambda q\,k_{n-1}(j,N)$ \bigskip

2. $k_n(j+1,N+1)=k_n(j,N)-\lambda p\,k_{n-1}(j,N)$\bigskip
\end{proposition}
These follow directly, first considering $(1+\lambda qv)G_N(v)=G_{N+1}(v)$ and second 
$$(1-\lambda pv)G_N(v)=G_{N+1}\bigm|_{j\to j+1}\ .$$
Note that the martingale property follows by combining $p$ times the first equation with $q$ times the second. \bigskip

Given four contiguous entries forming a $2\times2$ submatrix of $\Phi^{(N)}$, the \textsl{square identity} produces the lower left corner
from the other three values. In terms of the $k$'s:
\begin{proposition}{\sl Square identity}
$$k_n( j,N)=\lambda p\,k_{n-1}( j,N)+\lambda q\,k_{n-1}( j+1,N)+k_n( j+1,N)$$
\end{proposition}
\begin{proof}Combine $p$ times the first equation above with $q$ times that same equation with $j\to j+1$.
Applying the martingale recurrence on the left-hand side yields
$$k_n( j,N)=p\,k_n( j,N)+q\,k_n( j+1,N)+\lambda pq\,k_{n-1}( j,N)+\lambda q^2\,k_{n-1}( j+1,N)$$
Subtracting off $p\,k_n( j,N)$ and dividing out a common factor of $q$ yields the result.
\end{proof}

\end{subsection}
\begin{subsection}{Orthogonality}
For orthogonality, we wish to show that $\avg G(v)G(w).$ is a function of the product $vw$ only. We have, using independence and centering,
\begin{align*}
\avg G(v)G(w).&=\prod \avg (1+(v+w)(X_i-\mu)+vw(X_i-\mu)^2). \\
&=\prod(1+vw\,\sigma^2)=(1+vw\,\sigma^2)^N
\end{align*}
where the variance $\sigma^2=\lambda^2 pq$ in our context. This yields the squared norms 
$$\|k_n\|^2=\avg k_n,k_n.=\binom{N}{n}(\lambda^2 pq)^n\ .$$

Introducing matrices, we can express the orthogonality relations compactly.
Let $B$, the \textsl{binomial distribution matrix}, be the diagonal matrix
$$B=\text{diag}\,(p^N,Np^{N-1}q,\ldots,\binom{N}{i}\,p^{N-i}q^i,\ldots,q^N)$$
Let $\Gamma$ denote the diagonal matrix of squared norms,
$$\Gamma=\text{diag}\,(1,N(\lambda^2 pq),\ldots,\binom{N}{i}\,(\lambda^2 pq)^i,\ldots,(\lambda^2 pq)^N)$$
For fixed $N$, we write $\Phi$ for $\Phi^{(N)}$ which has $ij$ entry equal to $k_i(j,N)$.
Now $G(v)=\sum v^i\Phi_{ij}$, and we have
\begin{align*}
\sum_{i,j} v^iw^j(\Phi B\Phi^T)_{ij}&=\sum_{i,j,n}v^i w^j\Phi_{in}B_{nn}\Phi_{jn}\\
&=\avg G(v)G(w).=\sum_n(vw)^n\Gamma_{nn}\ .
\end{align*}

In other words, the orthogonality relation takes the form

$$\Phi B\Phi^T=\Gamma$$
which gives for the inverse
$$\Phi^{-1}=B\Phi^T\Gamma^{-1}\ .$$
\end{subsection}

In the following sections we will detail linearization formulas for the symmetric and non-symmetric cases, derive the corresponding
recurrence formulas and then look at the associated convolution operators on functions.
\end{section}

\begin{section}{Krawtchouk polynomials: general case}
We have the generating function
$$
G(v)=(1+\lambda qv)^{N-j}(1-\lambda pv)^j=\sum_{0\le n\le N} v^n\,k_n(j,N)
$$
with $j$ running from $0$ to $N$. The main feature is the relation
$$G(v)=\prod (1+v(X_i-\mu)) $$
where $X_i$ are independent Bernoulli variables taking values $\lambda$ and $0$ with probabilities $p$ and $q$ respectively.

\begin{subsection}{Linearization coefficients}
We want the expansion of the product $k_\ell k_m$ in terms of $k_n$. First, a simple lemma
\begin{lemma} Let $X$ take values $\lambda$ and $0$. Then the identity
$$(X-\lambda p)^2=\lambda (q-p)(X-\lambda p)+\lambda^2pq$$
holds.
\end{lemma}
\begin{proof} It is immediately checked. To derive it, expand $x(x-\lambda )$ in Taylor series about $\lambda p$ and equate the result to zero.
\end{proof}
In our context, we can write this as
\begin{equation}\label{variance}
(X-\mu)^2=\lambda (q-p)(X-\mu)+\sigma^2
\end{equation}
Now multiply
\begin{align*}
G(v)G(w)&=\prod \bigl(1+(v+w)(X_i-\mu)+vw(X_i-\mu)^2\bigr)\\
&=\prod \bigl(1+(v+w+\lambda (q-p)vw)(X_i-\mu)+\sigma^2vw\bigr)
\end{align*}
by the Lemma. Factoring out $1+\sigma^2vw$ from each term and re-expanding yields
\begin{align}\label{eq:convo}\nonumber
\sum_{\ell,m} v^\ell w^m &k_\ell(j,N)k_m(j,N)\\&=(1+\sigma^2vw)^N\,\prod\bigl(1+\frac{v+w+\lambda (q-p)vw}{1+\sigma^2vw} (X_i-\mu)  \bigr) \nonumber\\
&=\sum_n(1+\sigma^2vw)^{N-n}(v+w+\lambda (q-p)vw)^n\,k_n(j,N)
\end{align}
Expanding the coefficient of $k_n$, we have
$$
\sum_{\alpha+\beta+\gamma=n,\delta}\binom{n}{\alpha,\beta,\gamma}\binom{N-n}{\delta}v^\alpha w^\beta (\lambda(q-p))^\gamma
(\sigma^2vw)^\delta
$$
Fixing 
$$\ell=n-\beta+\delta\quad\text{and}\quad m=n-\alpha+\delta$$
 yields
\begin{theorem}\label{thm:LF} {\rm Linearization formula.}\hfb
The coefficient of $k_n$ in the expansion of the product $k_\ell k_m$ is
$$\sum_\delta \frac{n!}{(n-m+\delta)!(n-\ell+\delta)!(\ell+m-n-2\delta)!}\,\binom{N-n}{\delta}  (\lambda(q-p))^{\ell+m-n-2\delta}
\sigma^{2\delta}\ .$$
\end{theorem}
\begin{subsubsection}{Recurrence formula}
The three-term recurrence formula characteristic of orthogonal polynomials follows by specializing $\ell=1$ in the linearization formula.
First, compute the constant term and coefficient of $v$ from the generating function $G$ :
$$k_0=1\quad\text{and}\quad k_1=\lambda(Nq-j)$$
From the linearization formula, we pick up three terms, with $n=m$ and $n=m\pm1$. We get
\begin{proposition}{\rm Recurrence formula}
$$\lambda(Nq-j)\,k_m=(m+1)k_{m+1}+\lambda (q-p)\,m\,k_m+\lambda^2pq(N+1-m)\,k_{m-1}$$ 
\end{proposition}
The terms $k_m$ and $k_{m+1}$ arise with $\delta=0$, with the term $k_{m-1}$ the only contribution for $\delta=1$.
\end{subsubsection}
\end{subsection}
\begin{subsection}{Krawtchouk transforms. Inversion}
We identify functions on $\{0,1,\ldots,N\}$ with $\RR^{N+1}$ and the Krawtchouk transforms via the action of the matrix
$\Phi^{(N)}$ on that space. For given $N$, $\Phi$ denotes $\Phi^{(N)}$. \bigskip

For our standard transform, we think of row vectors with multiplication by $\Phi$ on the right. Thus,
the transform $F$ of a function $f$ is given by
$$F(j)=\sum_n f(n)\,k_n(j,N) \qquad \text{or} \qquad  \=F=^\dag=\=f=^\dag\,\Phi$$
where, e.g., $\=f=$ is the column vector with entries the corresponding values of $f$. \bigskip

The inversion formula is conveniently expressed in terms of matrices. 
\begin{proposition}\label{prop:inversion}
Let $P$ be the diagonal matrix 
$$P=\mbox{\rm diag}\,((\lambda p)^N,\ldots,(\lambda p)^{N-j},\ldots,1)$$
Let $P'$ be the diagonal matrix
$$P'=\text{\rm diag}\,(1,\ldots,(\lambda p)^j,\ldots,(\lambda p)^N)$$
Then
$$\Phi P\Phi=\lambda^NP'$$
\end{proposition}

The proof is similar to that for orthogonality.

\begin{proof}
The matrix equation is the same as the corresponding identity via generating functions. Namely,
$$\sum_{i,j,n}v^i k_{i}(n,N)(\lambda p)^{N-n}k_{n}(j,N)w^j\binom{N}{j}=\lambda^N(1+\lambda pvw)^N$$
First, sum over $i$, using the generating function $G(v)$, with $j$ replaced by $n$. Then sum over $n$, again using the generating 
function. Finally, summing over $j$ using the binomial theorem yields the desired result, via $p+q=1$.
\end{proof}

Thus,

\begin{corollary}\label{cor:inversion}
$$\Phi^{-1}=\lambda^{-N}\,P\Phi P'^{-1}$$
\end{corollary}
which is the basis for an efficient inversion algorithm, being a simple modification of the original transform.
\end{subsection}
\begin{subsection}{Convolution}
Corresponding to the product of two transforms $F$ and $G$ is the convolution of the original functions $f$ and $g$.
We have, following the proof of the linearization formula, eqs. \eqref{eq:convo},
\begin{align*}
F(j)G(j)&=\sum_{\l,m}f(\l)g(m)k_\l(j)k_m(j)\\
&=\sum_n k_n(j) \sum_{\alpha,\beta,\delta}
 \frac{n!}{\alpha!\beta!(n-\alpha-\beta)!}\,\binom{N-n}{\delta} \\ &\qquad \times(\lambda(q-p))^{n-\alpha-\beta}(\lambda^2pq)^{\delta}
f(n-\beta+\delta)g(n-\alpha+\delta)  \ .
\end{align*}
Thus, we may define the convolution of two functions $f$ and $g$ on $\{0,1,\ldots,N\}$ by
\begin{align}\label{eq:conv}
 \nonumber (f\star g)(n)&=\sum_{\alpha,\beta,\delta} 
 \binom{n}{\alpha,\beta,n-\alpha-\beta}\,\binom{N-n}{\delta} \\&\qquad\times(\lambda(q-p))^{n-\alpha-\beta}(\lambda^2pq)^{\delta}
f(n-\beta+\delta)g(n-\alpha+\delta) 
\end{align}
and we have the relation
$$ F(j)G(j)=\sum_n (f\star g)(n)\,k_n(j) \ .$$
Now, using the inversion formula, Corollary  \ref{cor:inversion}, we have the relation
$$(f\star g)(n)=\lambda^{-N}\,\sum_j F(j)G(j)\,(\lambda p)^{N-n-j}\,k_j(n)$$
for the convolution of functions.
\end{subsection}
\end{section}

\begin{section}{Krawtchouk polynomials: symmetric case}
For the symmetric case, it is convenient to consider the ``signs" process where $X_i$ takes values $\pm1$ with equal probability,
$p=q=1/2$. Thus, $\lambda=2$ and we have the generating function
$$
G(v)=(1+2qv)^{N-j}(1-2pv)^j=(1+v)^{N-j}(1-v)^j=\sum_n v^n\,k_n(j,N) \ .
$$
While $j$ runs from $0$ to $N$, the sum $x=N-2j$ runs from $-N$ to $N$ in steps of 2.
Now, $q-p=0$ and $\sigma^2=1$. \bigskip

In terms of $x=k_1=N-2j$, write $k_n(j,N)=K_n(x,N)/n!$. We have the recurrence 
$$x\,K_n=K_{n+1}+n(N+1-n)\,K_{n-1}$$
with initial conditions $K_0=1$, $K_1=x$. For example, we can generate the next few polynomials
$$K_2=x^2-N\,,\quad K_3=x^3+(2-3N)x\,,\quad K_4=x^4+(8-6N)x^2+3N^2-6N \ .$$

The special identities and recurrences hold with $\lambda=2$, $p=q=1/2$ and simplify accordingly. Of particular interest 
is the simplification of the convolution structure.

\begin{subsection}{Linearization coefficients}
We want the expansion of the product $k_\ell k_m$ in terms of $k_n$. 
In Theorem \ref{thm:LF}, since $q=p$, we have the condition
$$\l+m-n=2\delta$$
and the sum over delta disappears. This leads to a particular set of conditions, namely, that the 
numbers $\l$, $m$, and $n$ satisfy the conditions that they should form the sides of a triangle.
So, define the \textit{triangle function}
$$\Delta(\l,m,n)=\frac{ (\frac{\l+m+n}{2})!}{(\frac{-\l+m+n}{2})!(\frac{\l-m+n}{2})!(\frac{\l+m-n}{2})!}$$
where all terms with a factorial must be nonnegative. Note that this is a multinomial coefficient.
\begin{proposition}\label{cc} In the symmetric case, the  expansion of the product $k_\ell k_m$ is
$$k_\l(j)\,k_m(j)=\sum\binom{n}{\frac{\l-m+n}{2}}\binom{N-n}{\frac{\l+m-n}{2}}\,k_n(j) \ .$$
Alternatively, we have
$$k_\l(j)\,k_m(j)=\sum\binom{N}{\frac{\l+m+n}{2}}\frac{\Delta(\l,m,n)}{\binom{N}{n}}\,k_n(j) \ .$$
\end{proposition}
\begin{remark}
If $\l+m\ge N$, then the two sides differ by a polynomial vanishing on the spectrum $\{-N,2-N,\ldots,N-2,N\}$.
\end{remark}
\begin{proof} 
The ``triangular" form follows from the binomial form by rearranging factorials.
\end{proof}
\end{subsection}
\begin{subsection}{Krawtchouk transforms. Inversion}
In the symmetric case, the matrices $P$ and $P'$ in Proposition \ref{prop:inversion} and Corollary \ref{cor:inversion}
become identity matrices. Thus, we have
$$\Phi^2=2^NI\qquad\text{and}\qquad\Phi^{-1}=2^{-N}\Phi\ .$$
So the inversion is essentially an immediate application of the original transform.
\end{subsection}

\begin{subsection}{Convolution}
Corresponding to the product of two transforms $F$ and $G$ is the convolution of the original functions $f$ and $g$.
In equation \eqref{eq:conv}, the condition $q-p=0$ entails $n=\alpha+\beta$. We write $a$ for $\alpha$, replacing
$\beta=n-a$ and write $b$ for $\delta$. This gives for the convolution
$$ (f\star g)(n)=\sum_{a,b}\binom{n}{a}\binom{N-n}{b} f(a+b)g(n-a+b)\ . $$
We have the relation
$$ F(j)G(j)=\sum_n (f\star g)(n)\,k_n(j)$$
and the inversion simplifies to
$$(f\star g)(n)=2^{-N}\,\sum_j F(j)G(j)\,k_j(n)$$
for the convolution of the original functions.
\end{subsection}
\end{section}

\begin{section}{Shift operators and matrix formulation of Krawtchouk transform and convolution}
We will show how the transform and convolution can be represented by matrices acting on appropriate spaces.
\begin{subsection}{Transforms}
Introduce the shift operator $T_x$ which acts on a function $f(x)$ by
$$T_xf(x)=f(x+1)$$
Similarly, $T_yf(y)=f(y+1)$ shifts the variable $y$ by 1. For the transform, in the generating function, we replace $v$ by $T_n$,
the matrix shifting the argument $n$ of $f$:
$$F(j)=\sum_n f(n)k_n(j)=\sum_n k_n(j)\,(T_n)^n f(0)=(1+\lambda qT_n)^{N-j}(1-\lambda p T_n)^jf(0)$$
Representing $f$ by the (column) vector of values 
$$\=f= =(f(0),f(1),\ldots,f(N))^\dag$$ 
$T_n$ is represented by the $N+1$ by $N+1$ matrix with 1's on the superdiagonal
and zeros elsewhere and the above formula can be computed recursively using matrices of a very simple form. 
The value $F(j)$ will be the top entry in the resulting vector at each step. \bigskip

One approach is to form
$$T(N)=(1+\lambda qT_n)^N\quad\text{and}\quad U=(1+\lambda qT_n)^{-1}(1-\lambda p T_n)$$
and compute successively 
\begin{equation}\label{eq:iter}
T(N)\=f=\,,UT(N)\=f=\,,U^2T(N)\=f=\,,\ldots,U^NT(N)\=f=
\end{equation}
Form a matrix with these vectors as columns. Then the entries along the top row are the values $F(j)$.  \bigskip

\begin{remark}
Even though we are using the column vector $\=f=$, we are taking the  \textit{transform} multiplying by $\Phi$ on the right, 
that is, computing the entries of $\=f=^\dag \Phi$.
\end{remark}

Considering vectors $\=f=$ with a single nonzero entry equal to one leads to another way to describe the result. Namely, 
the matrix with successive columns equal to the first row from each of the generated matrices $U^jT(N)$ produces $\Phi$.
(See Appendix.) \bigskip

\begin{remark} Note that the matrix $U$ has the expansion
$$U=I-\lambda T+\lambda^2qT^2-\lambda^3q^2T^3+\cdots\;\ .$$
This follows from the identity
\begin{equation}\label{eq:UandT}
I-U=\lambda T(I+\lambda qT)^{-1}
\end{equation}
which may be verified by multiplying both sides by $(I+\lambda qT)$. Expanding in geometric series, noting that $T$ is nilpotent, yields
the above formula for $U$. The coefficients are the entries constant on successive superdiagonals of $U$. 
\end{remark}

\begin{example} Let $N=4$. We have
$$T(4)=(I+\lambda q T)^4= \left[ \begin {array}{ccccc} 1&4\,\lambda\,q&6\,{\lambda}^{2}{q}^{2}&
4\,{\lambda}^{3}{q}^{3}&{\lambda}^{4}{q}^{4}\\0&1&4
\,\lambda\,q&6\,{\lambda}^{2}{q}^{2}&4\,{\lambda}^{3}{q}^{3}
\\0&0&1&4\,\lambda\,q&6\,{\lambda}^{2}{q}^{2}
\\0&0&0&1&4\,\lambda\,q\\0&0&0&0&1
\end {array} \right]$$
and
$$U=(1+\lambda qT)^{-1}(1-\lambda p T)=\left[ \begin {array}{ccccc} 1&-\lambda&{\lambda}^{2}q&-{\lambda}^{3}
{q}^{2}&{\lambda}^{4}{q}^{3}\\0&1&-\lambda&{\lambda}
^{2}q&-{\lambda}^{3}{q}^{2}\\0&0&1&-\lambda&{\lambda
}^{2}q\\0&0&0&1&-\lambda\\\noalign{\medskip}0&0&0&0&
1\end {array} \right]$$

Starting with a column vector $\=f=$, first multiplying by $T(4)$, then successively by $U$ produces one-by-one the entries
of the transform of $\=f=$. \bigskip

For the symmetric case, $\lambda=2$, $p=q=1/2$, $T(N)$ has the binomial coefficients along the superdiagonals while, except for 
1's on the diagonal, the entries of $U$ are $\pm2$ on alternating superdiagonals. Thus,

$$T(4)=(I+ T)^4=\left[ \begin {array}{ccccc} 1&4&6&4&1\\0&1&4&6&4
\\0&0&1&4&6\\0&0&0&1&4
\\0&0&0&0&1\end {array} \right]\quad\text{and}\quad
U= \left[ \begin {array}{ccccc} 1&-2&2&-2&2\\0&1&-2&2&
-2\\0&0&1&-2&2\\0&0&0&1&-2
\\0&0&0&0&1\end {array} \right]\ .
$$

\end{example}

Similarly, replacing the variables $v$ and $w$ in equation \eqref{eq:convo}, by $T_n$ and $T_m$ respectively yields the formula
\begin{align*}
\sum_{n,m}k_n(j,N)&k_m(j,N) (T_n)^n (T_m)^m f(0)g(0)\\
&=\sum_n k_n(j,N) \,(1+\sigma^2T_nT_m)^{N-n}(T_n+T_m+\lambda (q-p)T_nT_m)^n f(0)g(0)
\end{align*}

Representing $T_mT_n$ by the Kronecker/tensor product of the corresponding shift matrices provides an explicit matrix
that when applied to the tensor product of the vectors $\=f=$ and $\=g=$ yields the convolution $f*g$. (See Appendix for an example.)\bigskip

So the convolution can be computed analogously to the transform. Start with
$$T(N)=(1+\sigma^2T_nT_m)^N\quad\text{and}\quad U=(1+\sigma^2T_nT_m)^{-1}(T_n+T_m+\lambda (q-p)T_nT_m)$$
and compute successively as in equation \eqref{eq:iter}. 
\end{subsection}
\end{section}

\begin{section}{Dual Transforms. Binomial bases}
Of course, one could define transforms dually by multiplying column vectors :
$$F(n)=\sum_j k_n(j) f(j) \qquad \text{or} \qquad \=F= =\Phi\=f= \ .$$
Let's begin with an example. \bigskip

\begin{example} For the symmetric case, we observe the result
$$\left[ \begin {array}{rrrrr} 1&1&1&1&1\\4&2&0&-2&-4
\\6&0&-2&0&6\\4&-2&0&2&-4
\\1&-1&1&-1&1\end {array} \right]
 \left[ \begin {array}{ccccc} 1&1&1&1&1\\0&1&2&3&4
\\0&0&1&3&6\\0&0&0&1&4
\\0&0&0&0&1\end {array} \right]
=
  \left[ \begin {array}{ccccc} 1&2&4&8&16\\4&6&8&8&0
\\6&6&4&0&0\\4&2&0&0&0
\\1&0&0&0&0\end {array} \right]  \ .
$$
The matrix on the left is $\Phi^{(4)}$. Observe that column, $m$, say, of binomial coefficients is mapped to its partner column indexed by 
$N-m$, scaled by $2^m$. Note that as functions, functions with zero tails are
mapped to functions with zero tails, analogously to Fourier transforms of compactly supported functions or cutoff functions for filtering. 
\end{example}

In general we have
\begin{proposition} Let $f_m(j)=\displaystyle \binom{m}{j}\,p^{m-j}q^j$. Then the dual transform, $\=F=_m =\Phi\=f=_m$, is given by
$$F_m(n)=\binom{N-m}{n}\,(\lambda q)^n \ .$$
\end{proposition}
\begin{proof}
We show the generating function version of the relation. Thus, 
\begin{align*}
\sum_j (1+\lambda qv)^{N-j}&(1-\lambda pv)^j \displaystyle \binom{m}{j}\,p^{m-j}q^j\\
&=(1+\lambda qv)^{N-m}(1+\lambda qv)^{m-j}(1-\lambda pv)^j\displaystyle \binom{m}{j}\,p^{m-j}q^j  \\
&=(1+\lambda qv)^{N-m}(q-\lambda pq v+p+\lambda pqv)^m\\
&=(1+\lambda qv)^{N-m}=\sum_n v^n\binom{N-m}{n}\,(\lambda q)^n \ .
\end{align*}\rightline\qedhere
\end{proof}

So in this basis, call it \textsl{the binomial basis}, $\Phi$ is represented by a matrix with entries on the antidiagonal. Continuing our example, write
$$\B=\left[ \begin {array}{ccccc} 1&1&1&1&1\\0&1&2&3&4
\\0&0&1&3&6\\0&0&0&1&4
\\0&0&0&0&1\end {array} \right] \qquad  \text{and} \qquad 
J=\left[ \begin {array}{ccccc} 0&0&0&0&1\\0&0&0&1&0\\0&0&1&0&0\\0&1&0&0&0
\\1&0&0&0&0\end {array} \right]$$
With $D$ the diagonal matrix $\diag{1,2,2^2,2^3,2^4}$, we have 
\begin{equation}\label{eq:JD}
\Phi \B =\B JD \qquad  \text{or} \qquad {\B}^{-1}\Phi \B=JD
\end{equation}
with 
$$JD = \left[ \begin {array}{ccccc} 0&0&0&0&16\\\noalign{\medskip}0&0&0&8&0
\\\noalign{\medskip}0&0&4&0&0\\\noalign{\medskip}0&2&0&0&0
\\\noalign{\medskip}1&0&0&0&0\end {array} \right]$$
 the matrix representing the transform in the binomial basis. These relations extend to all $N$. \bigskip

A related family of transforms is indicated by the similar calculation
$$\left[ \begin {array}{ccccc} 1&0&0&0&0\\4&1&0&0&0
\\6&3&1&0&0\\4&3&2&1&0
\\1&1&1&1&1\end {array} \right]\Phi= \left[ \begin {array}{ccccc} 1&1&1&1&1\\8&6&4&2&0
\\24&12&4&0&0\\32&8&0&0&0
\\16&0&0&0&0\end {array} \right]$$
which can be expressed in the form
$$J\B J\Phi=D\B J \qquad  \text{or} \qquad (J{\B}J)\Phi (J\B J)^{-1}=DJ$$
with $J$, $D$, and $\B$ as above. Comparing with equation \eqref{eq:JD} indicates a connection
between $\Phi$ and $\Phi^T$. At this point it is straightforward to give a direct proof of the properties
we want.

\begin{proposition} Let $f_i(n)=\displaystyle \binom{N-n}{N-i}\,p^{i-n}\lambda^{-n}$. Then the transform, $\=F=_i^\dag =\=f=_i^\dag\Phi$, is given by
$$F_i(j)=\binom{N-j}{i}\ .$$
\end{proposition}
\begin{proof}
As in the previous proposition, we show the generating function version of the relation. Consider
\begin{align*}
\sum_i\sum_n \binom{N-n}{N-i}\,p^{i-n}&\lambda^{-n}k_n(j)v^i
=\sum_i\sum_n \binom{N-n}{i} p^{N-i} (\lambda p)^{-n}k_n(j) v^{N-i}\\
&=\sum_n (pv)^n\sum_i \binom{N-n}{i} (vp)^{N-n-i} (\lambda p)^{-n}k_n(j) \\
&=\sum_n (pv)^n(1+pv)^{N-n}(\lambda p)^{-n}\,k_n(j)\\
&=(1+pv)^N\sum_n \left(\frac{v/\lambda}{1+pv}\right)^n\,k_n(j)\\
&=(1+pv)^N\,\left(1+\frac{qv}{1+pv}\right)^{N-j}\left(1-\frac{pv}{1+pv}\right)^{j}\\
&=(1+v)^{N-j}=\sum_i \binom{N-j}{i} v^i \ .
\end{align*}\rightline\qedhere
\end{proof}
\end{section}

\begin{section}{Concluding Remarks}
We have presented Krawtchouk transforms which have the potential to provide
an inherently discrete, efficient alternative to Fourier analysis. By presenting effective algorithms
using matrix techniques to compute transforms and convolution products, we have demonstrated
useful tools that are not only of theoretical interest but are ready for practical applications.
As well, the special binomial transforms we have indicated provide a solid basis for filtering techniques.
Thus, the Krawtchouk analogs of the standard Fourier toolkit are now available. 
Digital image analysis, for example, will provide an important arena for illustrating and developing Krawtchouk methods
as presented in this work. 
\end{section}
\begin{section}{Appendix}
Here we show examples of a transform and of a convolution computation using the matrix techniques discussed in the text.
\begin{subsection}{Krawtchouk transform}
For a non-symmetric example, we take $N=4$, $\lambda=2$, $p=1/4$. We have
$$ T(4)=\left[ \begin {array}{ccccc} 1&6&{\frac {27}{2}}&{\frac {27}{2}}&{\frac {81}{16}}\phantom{\Bigm|}\\
0&1&6&{\frac {27}{2}}&{\frac {27}{2}}\phantom{\Bigm|}\\
0&0&1&6&{\frac {27}{2}}\phantom{\Bigm|}\\
0&0&0&1&6\mathstrut\\
0&0&0&0&1\end {array} \right] 
\,,\ 
U= \left[ \begin {array}{ccccc} 1&-2&3&-9/2&27/4\\
0&1&-2&3&-9/2\\
0&0&1&-2&3\\
0&0&0&1&-2\\
0&0&0&0&1
\end {array} \right]  \ .
$$
The matrices $U^jT(4)$, $0\le j\le 4$, are successively generated, yielding 
$$
 \left[ \begin {array}{ccccc} 1&6&{\frac {27}{2}}&{\frac {27}{2}}&{\frac {81}{16}}\phantom{\Bigm|}\\
0&1&6&{\frac {27}{2}}&{\frac {27}{2}}\phantom{\Bigm|}\\
0&0&1&6&{\frac {27}{2}}\phantom{\Bigm|}\\
0&0&0&1&6\mathstrut\\
0&0&0&0&1\end {array} \right] \,,\ 
\left[ \begin {array}{ccccc} 1&4&\frac{9}{2}&0&-{\frac {27}{16}}\phantom{\Bigm|}\\
0&1&4&\frac{9}{2}&0\phantom{\Bigm|}\\
0&0&1&4&\frac{9}{2}\phantom{\Bigm|}\\
0&0&0&1&4\\
0&0&0&0&1\end {array}
 \right] ,
 \left[ \begin {array}{ccccc} 1&2&-\frac12&-\frac{3}{2}&{\frac {9}{16}}\phantom{\Bigm|}\\
0&1&2&-\frac12&-\frac{3}{2}\phantom{\Bigm|}\\
0&0&1&2&-\frac12\phantom{\Bigm|}\\
0&0&0&1&2\\0&0&0&0&1\end {array}
 \right] 
$$
and
$$
\left[ \begin {array}{ccccc} 1&0&-\frac{3}{2}&1&-\frac{3}{16}\phantom{\Bigm|}\\
0&1&0&-\frac{3}{2}&1\phantom{\Bigm|}\\
0&0&1&0&-\frac{3}{2}\\0&0&0&1&0\phantom{\Bigm|}\\
0&0&0&0&1\end {array} \right] \,,\ 
\left[ \begin {array}{ccccc} 1&-2&\frac{3}{2}&-\frac12&\frac{1}{16}\phantom{\Bigm|}\\
0&1&-2&\frac{3}{2}&-\frac12\phantom{\Bigm|}\\
0&0&1&-2&\frac{3}{2}\phantom{\Bigm|}\\
0&0&0&1&-2\\
0&0&0&0&1\end {array} \right]  \ .
 $$
Think of applying each of these matrices to the column vector consisting of all zeros except for 1 in the $j^{\rm th}$ spot. Then
the transform would be the $j^{\rm th}$ column of $\Phi$. These come from the successive entries in the top row, column $j$. In other words,
the top row of $U^jT(N)$ is the $j^{\rm th}$ column of $\Phi$. Concatenating the  transposed first rows yields
$$\Phi=\left[ \begin {array}{ccccc} 1&1&1&1&1\\
6&4&2&0&-2\\27/2&9/2&-1/2&-3/2&3/2\\
27/2&0&-3/2&1&-1/2\\
81/16&-27/16&9/16&-3/16&1/16
\end {array} \right] \ .
$$

Forming the diagonal matrices $P$ and $P'$, as in Prop.\,\ref{prop:inversion}, pre-multiplying by $\lambda^{-4}P$ and post-multiplying
by $P'^{-1}$ yields $\Phi^{-1}$ immediately.
\end{subsection}
\begin{subsection}{Krawtchouk convolution}
For an example of convolution, take $N=2$, $\lambda=2$, $p=1/4$, $\sigma^2=3/4$. We have the formulas
$$T(2)=(1+\sigma^2T_nT_m)^2\,\  \text{and}\ U=(1+\sigma^2T_nT_m)^{-1}(T_n+T_m+\lambda(q-p)T_nT_m) \ .$$
Let $T= \left[ \begin{smallmatrix} 0&1&0\\0&0&1\\0&0&0\end{smallmatrix}\right]$.
Set 
$$T_n=T\otimes I\ \text{and}\ T_m=I\otimes T$$
 with $I$ the $3\times3$ identity.   \bigskip

[Note the tensor sign denotes Kronecker product associated to the left.] \bigskip

Then form $T(2)$ and $U$. 
As for the transform, calculate $U^jT(N)$ successively. Here we show the top row(s) only, stacked to form a matrix
$$ \left[ \begin {array}{ccccccccc} 1&0&0&0&3/2&0&0&0&9/16\\ 0&1&0&1&1&3/4&0&3/4&3/4\\ 
0&0&1&0&2&2&1&2&1\end {array} \right] 
$$
Multiplying the column $\=f= \otimes \=g=$ on the left by the above matrix produces the convolution $f\star g$:
$$ \left[ \begin {array}{c} {f(0)}\,{g(0)}+\frac{3}{2}\,{f(1)}\,{g(1)}+
{\frac {9}{16}}\,{f(2)}\,{g(2)}\phantom{\Bigm|}\\
 {f(0)}\,{g(1)}+{f(1)}\,{g(0)}+{f(1)}\,{g(1)}+\frac{3}{4}\,{f(1)}\,{g(2)}+\frac{3}{4}
\,{f(2)}\,{g(1)}+\frac{3}{4}\,{f(2)}\,{g(2)}\phantom{\Bigm|}\\
 {f(0)}\,{g(2)}+2\,{f(1)}\,{g(1)}+2\,{f(1)}\,{g(2)}+{f(2)}\,{g(0)}+2\,{f(2)}\,{g(1)}+{f(2)}\,{g(2)}\end {array} \right]
$$
As in the previous section, we can compute
$$\Phi =\left[ \begin {array}{ccc} 1&1&1\\3&1&-1\\9/4&-3/4&1/4\end {array} \right] \ .$$
We have $\=F=^\dag=\=f=^\dag\,\Phi$ and $\=G=^\dag=\=g=^\dag\,\Phi$, for example,
$$\=F=^\dag= \left[{f(0)}+3\,{f(1)}+9/4\,{f(2)},{f(0)}+{f(1)}-3/4\,{f(2)},{f(0)}-{f(1)}+1/4\,{f(2)}\right]$$
and similarly for $\=G=$. One verifies that the $j^{\rm th}$ component of $({\bf {f\star g}}) ^\dag\,\Phi$ is indeed $F(j)G(j)$.
\end{subsection}

\end{section}

\end{document}